\documentclass{asl}

\usepackage{enumerate}
\usepackage[all]{xy}

\makeatletter
\def\imod#1{\allowbreak\mkern10mu({\operator@font mod}\,\,#1)}
\makeatother

\newtheorem{theorem}{Theorem}

\theoremstyle{definition}

\theoremstyle{remark}
\newtheorem{remark}{Remark}

\theoremstyle{remark}

\numberwithin{equation}{section}

    \DeclareMathOperator{\proj}{proj}

    \newcommand{\pinf}{[\omega]^\omega}

\title[$\Sigma^1_2$ and $\Pi^1_1$ mad families]{$\Sigma^1_2$ and $\Pi^1_1$ mad families}

\author{Asger T\"ornquist}
\revauthor{T\"ornquist, Asger}

\address{Department of Mathematical Sciences, University of Copenhagen, Universitetsparken 5, 2100 Copenhagen, Denmark}

\email{asger@logic.univie.ac.at}

\thanks{Research supported by Denmark's Council for Independent Research (Natural Sciences Division), grant no. 10-082689/FNU.}
 
\date{April 22, 2013}
\begin{document}

\begin{abstract}
We answer in the affirmative the following question of J\"org Brendle: If there is a $\Sigma^1_2$ mad family, is there then a $\Pi^1_1$ mad family? 
\end{abstract}
\maketitle

J\"org Brendle asked his question about $\Sigma^1_2$ and $\Pi^1_1$ mad families in his talk (see \cite{oberwolfach}) at the Set Theory meeting at Oberwolfach in January, 2011. The question was motivated by the recent body of work on the structure of definable mad families by an array of authors, including Fischer, Friedman and Zdomskyy \cite{frzd10, fifrzd11}, as well as Brendle and Khomskii, \cite{brkh}, who have later used the solution given below in their work.

\begin{theorem}
If there is a $\Sigma^1_2(a)$ mad family then there is a $\Pi^1_1(a)$ mad family.
\end{theorem}

\begin{proof}
The proof is short and simple. For $x\in \pinf$, let $x_n$ denote the $(n+1)$st element of $x$. We define two functions $g_i:(\pinf)^2\to \mathcal P(3\times\omega)$, $i\in\{0,1\}$, as follows:
\begin{align*}
&g_0(x,y)=\{0\}\times x\cup \{2\}\times\{x_n:n\in y\},\\
&g_1(x,y)=\{1\}\times x\cup \{2\}\times\{x_n:n\notin y\}.
\end{align*}
Now let $\mathcal A_0\subseteq\pinf$ be a $\Sigma^1_2$ mad family, and let $F_0\subseteq (\pinf)^2$ be $\Pi^1_1$ such that $\proj F_0=\mathcal A_0$. By uniformization, we may assume that $F_0$ is the graph of a partial function. Define
$$
\mathcal A=g_0(F_0)\cup g_1(F_0).
$$
It is clear that $\mathcal A$ is an almost disjoint family of subsets of $3\times\omega$. Further, $\mathcal A$ is a mad family. To see this, let $p_i:3\times\omega\to\omega$ be the projection map, $i\in\{0,1,2\}$. If $z\subseteq 3\times\omega$ is an infinite set then there is $i\in 3$ such that $p_i(z)$ is infinite. Let $x\in\mathcal A_0$ be such that $p_i(z)\cap x$ is infinite, and let $y\in\pinf$ be such that $(x,y)\in F_0$. If $i\in\{0,1\}$ we have $|z\cap g_i(x,y)|=\aleph_0$. If $i=2$ then $p_2(g_0(x,y)\cup g_1(x,y))=x$, and so we must have that either $|p_2(g_0(x,y))\cap z|=\aleph_0$ or $|p_2(g_1(x,y))\cap z|=\aleph_0$.

We claim that $\mathcal A$ is $\Pi^1_1$. To see this, let $z\in\mathcal A$, and assume that there is $(x,y)\in F_0$ such that $g_0(x,y)=z$. Note that $p_0(z)=x$, and that from $p_2(z)$ and $p_0(z)$ we clearly can recover $y$ (by using the inverse of the function $n\mapsto x_n$) in a recursive way. A similar argument applies when $z=g_1(x,y)$. Thus
$$
z\in\mathcal A\iff (\exists x,y\in\Delta^1_1(z)) (F_0(x,y)\wedge g_0(x,y)=z)\vee F_0(x,y)\wedge g_1(x,y)=z),
$$
which is a $\Pi^1_1$ definition of $\mathcal A$ by the Spector-Gandy theorem (see e.g. \cite[Corollary 4.19]{mawe85}).

Since the above proof clearly relativizes to a parameter $a$, we are done.
\end{proof}

\begin{remark}
In \cite{miller89}, Arnold Miller proved, using a coding argument, that if $V=L$ then there is a $\Pi^1_1$ mad family. Since it is routine to check that $V=L$ implies the existence of a $\Sigma^1_2$ mad family, the above also provides a different proof of Miller's theorem.
\end{remark}

\bibliographystyle{asl}
\bibliography{sigma_12mad}

\end{document}